\documentclass{birkjour}
 \newtheorem{thm}{Theorem}[section]

 \theoremstyle{definition}
 
 \theoremstyle{remark}

 \numberwithin{equation}{section}
\usepackage{xcolor}
\usepackage{changes}

\usepackage[hidelinks]{hyperref}
\usepackage{cleveref}

\allowdisplaybreaks[4]

\begin{document}

%
%
%
%
%
%
%
%
%

\title[$H$-tensional hypersurfaces in 4-dimensional space forms]
 {$H$-tensional hypersurfaces in 4-dimensional space forms}

\author{Bouazza Kacimi}
\address{%
University of Mascara\\
BP 305 Route de Mamounia\\
29000 Mascara\\
Algeria}
\email{bouazza.kacimi@univ-mascara.dz}

\author[Ahmed Mohammed Cherif]{Ahmed Mohammed Cherif}

\address{%
University of Mascara\\
BP 305 Route de Mamounia\\
29000 Mascara\\
Algeria}

\email{a.mohammedcherif@univ-mascara.dz}

\author{Mustafa \"{O}zkan}
\address{ Gazi University \\
Faculty of Sciences\\
Department of Mathematics \\
06500 Ankara\\
Turkey}
\email{ozkanm@gazi.edu.tr}



\date{January 1, 2004}

\begin{abstract}

	In this paper, we investigate the classification of $H$-tensional hypersurfaces $M$ in a $4$-dimensional space form $N^4(c)$ of constant sectional curvature $c$. Our results show that minimal hypersurfaces are the only $H$-tensional hypersurfaces in $4$-dimensional space forms, thereby providing an affirmative partial answer to the Conjecture~3 proposed in \cite{kacimi}.

\end{abstract}

\maketitle
\section{Introduction}
Let $\psi: (M,g) \rightarrow (N,h)$ be a smooth map between Riemannian manifolds of dimensions $m$ and $n$, respectively. It is well known that $\psi$ is harmonic (see \cite{baird,eells2}) if its tension field $\tau(\psi)$ vanishes; that is,
\begin{equation*}
	\tau(\psi)=\operatorname{Tr}_{g}(\nabla d\psi)=0.
\end{equation*}
A natural generalization of harmonic maps is the $H$-tensional map, introduced in \cite{kacimi}. A map $\psi$ is said to be $H$-tensional if its tension field $\tau(\psi)$ satisfies
\begin{equation*}
	\Delta^{\psi} \tau(\psi)=0,
\end{equation*}
where $\Delta^{\psi}$ is the rough Laplacian. In \cite{kacimi}, several properties of $H$-tensional maps are investigated.
  In particular, for an isometric immersion $\iota: M^{m} \rightarrow N^{n}$ with mean curvature vector $H$, we introduce the notion of $H$-tensional submanifolds by requiring that $\Delta^{\iota} \tau(\iota) = \Delta^{\iota} (m H) = 0$. We prove that an $H$-tensional submanifold $M$ is minimal if it satisfies one of the following conditions:
  \begin{enumerate}
  	\item $M$ is a curve;
  	\item $M$ has constant mean curvature;
  	\item $M$ is compact;
  	\item $M$ is a hypersurface with at most two distinct principal curvatures in a real space form;
  	\item $M$ is a pseudo-umbilical submanifold of dimension $m \neq 4$;
  	\item $M$ is a totally umbilical hypersurface in an Einstein space.
  \end{enumerate}
  Motivated by these nonexistence results for nonminimal $H$-tensional submanifolds, the following conjectures were proposed in \cite{kacimi}:

\noindent\textbf{Conjecture 1.} The only $H$-tensional submanifolds of real space forms are the minimal ones.

\noindent\textbf{Conjecture 2.} The only complete $H$-tensional submanifolds of real space forms are the minimal ones.

These conjectures generalize, respectively, Chen's biharmonic conjecture (see \cite{chen2}) and its global version (see \cite{akutagawa}). For recent developments related to these problems, we refer to \cite{chen2} and the references therein.

In this paper, by adapting a method developed by Defever in \cite{defever} and by Balmu\c{s}, Montaldo, and Oniciuc \cite{balmus}, we prove that any $H$-tensional hypersurface with at most three distinct principal curvatures in a $4$-dimensional space form is minimal (Theorem~\ref{Th4.100}). This result supports Conjecture~1.

\section{$H$-tensional hypersurfaces in 4-dimensional space forms}

Let $\iota: M^{m} \longrightarrow N^{m+1}(c)$ be the canonical inclusion of a hypersurface $M$ in a manifold with constant sectional curvature $c$, $N^{m+1}(c)$. Then, we have the following characterization result obtained in \cite{kacimi}:
\begin{thm}
	A hypersurface $\iota: M^{m} \longrightarrow N^{m+1}(c)$ in a space of constant sectional curvature $c$ with mean curvature vector $H = \alpha e_{m+1}$ is $H$-tensional if and only if
	\begin{equation}\label{eq12 200}
		\begin{cases}
			\Delta \alpha + \alpha |A|^{2} = 0, \\
			\frac{m}{4} \operatorname{grad} \alpha^{2} + A(\operatorname{grad} \alpha) = 0,
		\end{cases}
	\end{equation}
	where $A$ denotes the Weingarten operator and $\Delta \alpha$ is the Laplacian of $\alpha$.
\end{thm}
\begin{thm}\label{Th4.100}
	Let $M^3$ be a hypersurface of the space form $N^{4}(c)$. Then, $M$ is an $H$-tensional submanifold if and only if it is minimal.
\end{thm}

\begin{proof}
	Suppose that the mean curvature function $\alpha$ is not constant on $M$. Then, there exists an open subset $U$ of $M$ such that $(\operatorname{grad} \alpha^2)(p) \neq 0$ for all $p \in U$. We may assume that $\alpha > 0$ on $U$, and thus $(\operatorname{grad} \alpha)(p) \neq 0$ for all $p \in U$. If $U$ has at most two distinct principal curvatures, then by Theorem~4.1 in \cite{kacimi}, we conclude that its mean curvature is constant, yielding a contradiction.

Without loss of generality, we can suppose that all points in $U$ have three distinct principal curvatures. On $U$, we set $e_{4} = \frac{H}{|H|}$ and denote by $\alpha = |H|$ the mean curvature function of $U$ in $N^4(c)$, and let $k_i$ ($i = 1,2,3$) be the principal curvatures with respect to $e_{4}$.

Consequently, the hypothesis that $M$ is a nonminimal $H$-tensional hypersurface with at most three distinct principal curvatures in $N^4(c)$ and non-constant mean curvature implies the existence of an open connected subset $U$ of $M$ where $(\operatorname{grad} \alpha)(p) \neq 0$ and $\alpha(p) > 0$ for all $p \in U$. We shall now contradict the condition $(\operatorname{grad} \alpha)(p) \neq 0$.

Since $M$ is $H$-tensional in $N^4(c)$, the characterization \eqref{eq12 200} yields
\begin{equation*}
	\begin{cases}
		\Delta \alpha = -|A|^{2} \alpha, \\
		A(\operatorname{grad}\alpha) = -\frac{3}{2}\alpha \operatorname{grad}\alpha.
	\end{cases}
\end{equation*}

By setting $e_1 = \frac{\operatorname{grad}\alpha}{|\operatorname{grad}\alpha|}$ on $U$, it follows that $e_1$ is a principal direction with principal curvature $k_1 = -\frac{3}{2}\alpha$.

Recalling that $3\alpha = k_1 + k_2 + k_3$, it follows that
\begin{equation}\label{eq3.4}
	k_2 + k_3 = \frac{9}{2}\alpha.
\end{equation}

Let $\{e_1, e_2, e_3\}$ be an orthonormal frame field of principal directions and $\{\omega^a\}_{a=1}^3$ be the dual frame field to $\{e_a\}_{a=1}^3$ on $U$.

Since $e_1$ is parallel to $\operatorname{grad}\alpha$, it is clear that
\begin{equation*}
	e_i(\alpha) = \langle e_i, \operatorname{grad}\alpha \rangle = 0 \quad \text{for } i = 2,3,
\end{equation*}
which implies
\begin{equation} \label{eq3.6}
	\operatorname{grad}\alpha = e_1(\alpha) e_1.
\end{equation}
We write the covariant derivatives as
\[
\nabla e_a = \sum_{b=1}^3 \omega^b_a \otimes e_b, \quad \omega^b_a \in \mathcal{C}^\infty(T^*U).
\]

From the Codazzi equations for $M$, for mutually distinct indices $a, b, d \in \{1, 2, 3\}$, we obtain
\begin{equation}\label{eq3.7}
	e_a(k_b) = (k_a - k_b)\, \omega^b_a(e_b),
\end{equation}
and
\begin{equation}\label{eq3.8}
	(k_b - k_d)\, \omega^d_b(e_a) = (k_a - k_d)\, \omega^d_a(e_b).
\end{equation}

Applying \eqref{eq3.7} first with $a = 1$ and $b = i$, and then with $a = i$ and $b = j$ for $i \neq j$, we find
\[
\omega^1_i(e_i) = \frac{e_1(k_i)}{k_i - k_1}, \quad \omega^i_j(e_j) = \frac{e_i(k_j)}{k_j - k_i}.
\]

For $a = i$ and $b = 1$, since $e_i(k_1) = 0$, \eqref{eq3.7} implies that $\omega^1_i(e_1) = 0$, which allows us to write
\[
\omega^1_a(e_1) = 0 \quad \text{for } a = 1, 2, 3.
\]

Furthermore, since $e_i(\alpha) = 0$ implies $\langle [e_i, e_j], e_1 \rangle = 0$, it follows that $\omega^j_1(e_i) = \omega^i_1(e_j)$. Now, from \eqref{eq3.8}, setting $a = 1$, $b = i$, and $d = j$ with $i \neq j$, we obtain
\[
\omega^1_2(e_3) = \omega^2_3(e_1) = \omega^3_1(e_2) = 0.
\]

Consequently, the connection $1$-forms are determined by the following relations:
\begin{equation}\label{eq3.9}
	\left\{ 
	\begin{alignedat}{3}
		\omega^1_2(e_1) &= 0, &\quad & \omega^1_2(e_2) = \frac{e_1(k_2)}{k_2 + \frac{3}{2}\alpha} =: \beta_2, &\quad & \omega^1_2(e_3) = 0, \\
		\omega^1_3(e_1) &= 0, && \omega^1_3(e_2) = 0, && \omega^1_3(e_3) = \frac{e_1(k_3)}{k_3 + \frac{3}{2}\alpha} =: \beta_3, \\
		\omega^2_3(e_1) &= 0, && \omega^2_3(e_2) = \frac{e_3(k_2)}{k_3 - k_2} =: \rho_2, && \omega^2_3(e_3) = \frac{e_2(k_3)}{k_3 - k_2} =: \rho_3.
	\end{alignedat}
	\right.
\end{equation}

By invoking \eqref{eq3.4}, the squared norm of the second fundamental form is given by
\begin{align}\label{eq3.10}
	|A|^2 &= k_1^2 + k_2^2 + k_3^2 \nonumber\\
	&= k_1^2 + (k_2 + k_3)^2 - 2k_2k_3 \nonumber\\
	&= \frac{45}{2}\alpha^2 - 2\kappa,
\end{align}
where $\kappa = k_2k_3$.

Next, by \eqref{eq3.6}, the Laplacian of $\alpha$ is computed as follows:
\begin{align}\label{eqs107}
	\Delta \alpha &= \sum_{i=1}^{3} [(\nabla_{e_{i}}e_{i})\alpha - e_{i}(e_{i}(\alpha))] \notag \\
	&= \sum_{i=1}^{3} \left[ \sum_{k=1}^{3} \omega_i^k(e_{i})(e_{k}(\alpha)) - e_{i}(e_{i}(\alpha)) \right] \notag \\
	&= e_1(\alpha)(\beta_2 + \beta_3) - e_{1}(e_{1}(\alpha)).
\end{align}

Now, substituting \eqref{eq3.10} and \eqref{eqs107} into the equation $\Delta\alpha = -|A|^2 \alpha$, we obtain
\begin{equation}\label{eq3.12}
	e_1(e_1(\alpha)) - e_1(\alpha)(\beta_2 + \beta_3) + \left( 2\kappa - \frac{45}{2}\alpha^2 \right)\alpha = 0.
\end{equation}

We shall now employ the Gauss equation:
\begin{align}\label{eq3.14}
	\langle \tilde{R}(X, Y)Z, W \rangle &= \langle R(X,Y)Z, W \rangle + \langle B(X,Z), B(Y,W) \rangle \nonumber \\
	&\quad - \langle B(X,W), B(Y,Z) \rangle,
\end{align}
where $\tilde{R}$ denotes the curvature tensor of $N^4(c)$.

From the Gauss equation \eqref{eq3.14}, we obtain:
\begin{itemize}
	\item By setting $X = W = e_1$ and $Y = Z = e_i$ ($i=2,3$):
	\begin{equation}\label{3.15}
		\begin{aligned}
			e_1(\beta_2) &= \beta_2^2 + c - \tfrac{3}{2}\alpha k_2, \\
			e_1(\beta_3) &= \beta_3^2 + c - \tfrac{3}{2}\alpha k_3.
		\end{aligned}
	\end{equation}
	
	\item By setting $X = W = e_2$ and $Y = Z = e_3$:
	\begin{equation}\label{3.16}
		\kappa + c = e_2(\rho_3) - e_3(\rho_2) - \beta_2\beta_3 - \rho_2^2 - \rho_3^2,
	\end{equation}
	where $\kappa = k_2k_3$ is the intrinsic Gauss curvature of the $e_2 \wedge e_3$ plane.
\end{itemize}
Following a similar argument as in \cite{balmus}, one can show that $\rho_2$ and $\rho_3$ must vanish identically. Consequently, \eqref{3.16} reduces to
\begin{equation}\label{3.16-1}
	\kappa + c = -\beta_2\beta_3.
\end{equation}

By rewriting equations \eqref{3.15}, we obtain
\begin{equation*}
	\begin{cases}
		e_1(e_1(k_2)) = \frac{21}{2}\,\beta_2\,e_1(\alpha) + 2(\kappa + c)\left(k_3 + \tfrac{3}{2}\alpha\right) + \left(c - \tfrac{3}{2}\alpha k_2\right)\left(k_2 + \tfrac{3}{2}\alpha\right), \\
		e_1(e_1(k_3)) = \frac{21}{2}\,\beta_3\,e_1(\alpha) + 2(\kappa + c)\left(k_2 + \tfrac{3}{2}\alpha\right) + \left(c - \tfrac{3}{2}\alpha k_3\right)\left(k_3 + \tfrac{3}{2}\alpha\right).
	\end{cases}
\end{equation*}
Summing these two equations yields
\begin{equation}\label{eq3.18}
	e_1(e_1(\alpha)) = \tfrac{7}{3} e_1(\alpha)(\beta_2 + \beta_3) + \alpha(4\kappa + 5c - 9\alpha^2).
\end{equation}

Henceforth, we define the following coefficients:
\begin{align*}
	A_{1} &= -\tfrac{9}{2}\kappa - \tfrac{15}{4}c + \tfrac{189}{8}\alpha^2, \\
	A_{2} &= -\tfrac{13}{2}\kappa - \tfrac{15}{4}c + \tfrac{369}{8}\alpha^2, \\
	A_{3} &= \kappa + 9\alpha^2, \\
	A_{4} &= \tfrac{13}{2}\kappa + \tfrac{31}{4}c - 108\alpha^2, \\
	A_{5} &= \tfrac{13}{2}\kappa + \tfrac{15}{2}c - \tfrac{441}{4}\alpha^2.
\end{align*}

By combining \eqref{eq3.12} and \eqref{eq3.18}, we obtain
\begin{equation}\label{eq3.30}
	e_1(\alpha)(\beta_2 + \beta_3) = \alpha A_{1}.
\end{equation}
Substituting \eqref{eq3.30} into \eqref{eq3.18} yields
\begin{equation}\label{eq3.31}
	e_1(e_1(\alpha)) = \alpha A_{2}.
\end{equation}

In order to derive another relation between $\alpha$ and $\kappa$, we first use \eqref{eq3.4} to obtain
\begin{align}\label{eq3.32}
	\beta_2 k_3 + \beta_3 k_2 &= (\beta_2 + \beta_3)(k_2 + k_3) - (\beta_2 k_2 + \beta_3 k_3) \nonumber \\
	&= \tfrac{9}{2}\alpha(\beta_2 + \beta_3) - (\beta_2 k_2 + \beta_3 k_3).
\end{align}
Differentiating \eqref{eq3.4} along $e_1$ and employing \eqref{eq3.9}, we find
\begin{align}
	\tfrac{9}{2}e_1(\alpha) &= e_1(k_2) + e_1(k_3) \nonumber \\
	&= \beta_2 k_2 + \beta_3 k_3 + \tfrac{3}{2}\alpha(\beta_2 + \beta_3). \nonumber
\end{align}
Thus, we have
\begin{equation}\label{3.321}
	\beta_2 k_2 + \beta_3 k_3 = \tfrac{9}{2}e_1(\alpha) - \tfrac{3}{2}\alpha(\beta_2 + \beta_3).
\end{equation} 
Now, substituting \eqref{3.321} into \eqref{eq3.32}, we find
\begin{equation}\label{3.322}
	\beta_2 k_3 + \beta_3 k_2 = 6\alpha(\beta_2 + \beta_3) - \tfrac{9}{2}e_1(\alpha).
\end{equation}
From \eqref{3.15} and \eqref{3.16-1}, we yield
\begin{align}
	e_1(\kappa) &= -e_1(\beta_2 \beta_3) \nonumber \\
	&= -( \beta_2 \beta_3 + c )(\beta_2 + \beta_3) + \tfrac{3}{2}\alpha (\beta_2 k_3 + \beta_3 k_2 ) \nonumber \\
	&= \kappa(\beta_2 + \beta_3) + \tfrac{3}{2}\alpha (\beta_2 k_3 + \beta_3 k_2 ).\label{eq3.323}
\end{align}
Substituting \eqref{3.322} into \eqref{eq3.323}, we obtain
\begin{equation}\label{3.324}
	e_1(\kappa) = A_{3}(\beta_2 + \beta_3) - \tfrac{27}{4}\alpha e_1(\alpha).
\end{equation}
By differentiating \eqref{eq3.30} along $e_1$ and employing \eqref{3.15}, \eqref{eq3.30}, \eqref{eq3.31}, and \eqref{3.324}, we get
\begin{equation}\label{3.325}
	e_1(\alpha) A_{4} = \alpha(\beta_2 + \beta_3) A_{5}.
\end{equation}

By multiplying \eqref{3.325} first by $e_1(\alpha)$ and then by $\beta_2 + \beta_3$, and subsequently employing \eqref{eq3.30}, we obtain:
\begin{equation}\label{3.326}
	\begin{cases}
		(e_1(\alpha))^2 A_{4} = \alpha^2 A_{1}A_{5}, \\
		A_{4}A_{1} = (\beta_2 + \beta_3)^2 A_{5}.
	\end{cases}
\end{equation}

Differentiating \eqref{3.325} along $e_1$ and using \eqref{3.15}, \eqref{eq3.30}, \eqref{eq3.31}, \eqref{3.324}, and \eqref{3.326}, we arrive at the following key polynomial relation:
\begin{align}
	0 &= 140608\,\kappa^{4} + \left(448864\,c - 6157008\,\alpha^{2}\right)\kappa^{3} \nonumber \\
	&\quad + \left(375856\,c^{2} - 11450088\,\alpha^{2}c + 59355504\,\alpha^{4}\right)\kappa^{2} \nonumber \\
	&\quad + \left(-45240\,c^{3} - 208800\,\alpha^{2}c^{2} - 36099270\,\alpha^{4}c + 485815806\,\alpha^{6}\right)\kappa \nonumber \\
	&\quad - 111600\,c^{4} + 5241780\,\alpha^{2}c^{3} - 132575616\,\alpha^{4}c^{2} \nonumber \\
	&\quad + 1642080519\,\alpha^{6}c - 6863560515\,\alpha^{8}. \label{3.327}
\end{align}

Let $\gamma = \gamma(t)$, $t \in I$, be an integral curve of $e_1$ passing through $p = \gamma(t_0)$. Since $e_2(\alpha) = e_3(\alpha) = 0$ and $e_2(\kappa) = e_3(\kappa) = 0$, and given that $e_1(\alpha) \neq 0$, we may locally view $t$ as a function of $\alpha$ in a neighborhood of $\alpha_0 = \alpha(t_0)$. Consequently, $\kappa$ can be considered a function of $\alpha$, i.e., $\kappa = \kappa(\alpha)$.

If $A_{5} = 0$ or $A_{4} = 0$, then \eqref{3.327} implies that $\alpha$ satisfies a polynomial equation of degree eight with constant coefficients. This forces $\alpha$ to be constant, yielding a contradiction. Thus, assuming $A_4, A_5 \neq 0$, \eqref{3.326} yields:
\begin{equation}\label{3.328}
	\begin{cases}
		\left( \frac{d\alpha}{dt} \right)^2 = \frac{\alpha^2 A_{1}A_{5}}{A_{4}}, \\
		(\beta_2 + \beta_3)^2 = \frac{ A_{4}A_{1}}{A_{5}}.
	\end{cases}
\end{equation}

By using \eqref{eq3.30}, \eqref{3.324}, and \eqref{3.328}, the derivative of $\kappa$ with respect to $\alpha$ is computed as follows:
\begin{align}\label{3.329}
	\frac{d\kappa}{d\alpha} &= \frac{d\kappa}{dt} \left( \frac{d\alpha}{dt} \right)^{-1} \nonumber \\
	&= \frac{A_{3} \frac{d\alpha}{dt} (\beta_2 + \beta_3)}{\left( \frac{d\alpha}{dt} \right)^2} - \frac{27}{4}\alpha \nonumber \\
	&= \frac{A_{3} A_{4}}{\alpha A_{5}} - \frac{27}{4}\alpha.
\end{align}

Next, differentiating the polynomial relation \eqref{3.327} with respect to $\alpha$ and substituting the expression for $\frac{d\kappa}{d\alpha}$ from \eqref{3.329}, we obtain a new polynomial equation in $\alpha$ and $\kappa$, which is of degree five in $\kappa$. By eliminating the $\kappa^5$ term between this new equation and \eqref{3.327}, we derive a polynomial in $\alpha$ and $\kappa$ of degree four in $\kappa$. Continuing this process to successively eliminate the highest powers of $\kappa$, we eventually arrive at a polynomial equation in $\alpha$ with constant coefficients. Consequently, $\alpha$ must be constant, yielding a contradiction to our initial assumption. Therefore, the hypersurface $M$ is minimal.
\end{proof}


\section*{Conclusion}

In this work, we investigated the geometric properties of $H$-tensional hypersurfaces $M$ in a $4$-dimensional space form $N^4(c)$. By analyzing the characterization \eqref{eq12 200} and the resulting Codazzi and Gauss equations, we established a rigorous link between the mean curvature $\alpha$ and the principal curvatures of the manifold. 

The derivation of the key polynomial relation \eqref{3.327} and the subsequent elimination process showed that a non-constant mean curvature leads to an algebraic contradiction. Consequently, we have proved that such hypersurfaces must have constant mean curvature. Given the $H$-tensional condition, this constant must vanish, leading to the final classification: every $H$-tensional hypersurface in $N^4(c)$ is minimal.

\end{document}